\documentclass[12pt]{article}
\usepackage[english]{babel}
\usepackage[tbtags]{amsmath}
\usepackage{amsfonts,amssymb,mathrsfs,amsmath,amscd,comment}

\usepackage{graphicx}
\graphicspath{{pictures/}} \DeclareGraphicsExtensions{.pdf,.png,.jpg}

\usepackage{amssymb}
\usepackage{amsthm}
\usepackage{amsmath}

\usepackage{todonotes}

\numberwithin{equation}{section}

\theoremstyle{plain}
\newtheorem{theorem}{Theorem}
\newtheorem*{theorem*}{Theorem}
\newtheorem{proposition}{Proposition}[section]
\newtheorem{lemma}{Lemma}[section]

\newtheorem*{lemma*}{Lemma}

\theoremstyle{definition}
\newtheorem*{defin}{Definition}
\newtheorem*{remark}{Remark}

\makeatletter
\def\blfootnote{\gdef\@thefnmark{}\@footnotetext}
\makeatother

\usepackage{amsmath,amssymb,amsfonts,amsthm,mathrsfs}

\usepackage{authblk}
\usepackage[left=2.5cm, right=3cm, top=2cm, bottom=3cm, bindingoffset=0cm, marginparwidth=3cm]{geometry}

\begin{document}

\title{Diophantine ``Tears of the Heart''}

\date{}

\author[1]{Yu.~S.~Ilyashenko}
\author[2]{S.~Minkov}
\author[1]{I.~Shilin}

\affil[1]{HSE University, Moscow, Russia}
\affil[2]{Brook Institute of Electronic Control Machines, Moscow, Russia}

\maketitle

\begin{flushright}
	\textit{To the memory of Valentin Senderovich Afraimovich,\\ an outstanding mathematician and an admirable person}
\end{flushright}

\begin{abstract}
Recent studies of topologically generic unfoldings of vector fields featuring a ``tears of the heart'' polycycle with one internal and one external winding separatrix have shown that, in a special one-parameter subfamily where the ``heart'' is preserved and the ``tear'' loop if broken, at least four invariants of weak topological classification appear. In this paper, we demonstrate that the metrical perspective yields a different result: for Lebesgue almost all values of the coefficients related to the original vector field, the special one-parameter family generates only two such invariants.

\bigskip


\end{abstract}

\blfootnote{This study was supported by the Basic Research Program of the HSE University (project No.~075-00648-25-00 “Symmetry. Information. Chaos.”).}

\maketitle

\section{Introduction}

In \cite{IKS}, an open set of unfoldings of the ``tears of the heart'' polycycle (see Fig.~\ref{fig:th}) on $\mathbb S^2$ was described. These unfoldings admit numeric invariants of topological conjugacy, the number of which depends on how many separatrices wind around the polycycle \cite{GK}.
\begin{figure}[!ht]
\begin{center}\label{fig:th}
\includegraphics{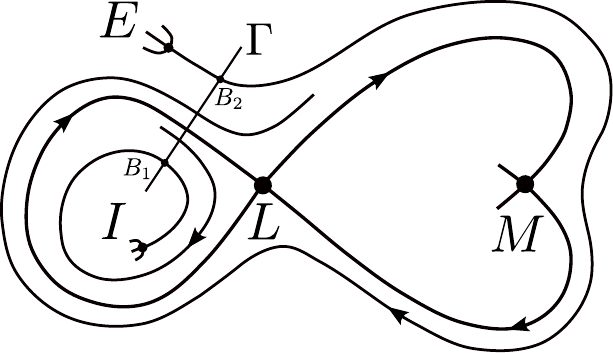}
\caption{ The polycycle ``tears of the heart'' with an exterior and interior saddles} 
\end{center}

 \end{figure}

In what follows, the simplest case of one separatrix winding inside the ``tear'' separatrix loop and one separatrix winding onto the whole polycyle is considered\footnote{As is standard, we assume that the vector field is structurally stable in restriction to the complement of a narrow neighborhood of the polycycle. We also adopt a technical `non-zero speed' condition (see, e.g., \cite{IMSh}). This assumption is needed solely to satisfy the hypotheses of the results we invoke from \cite{IKS, IMSh} and holds on an open dense set of families.}. We denote the two saddles of the polycycle by $L$ and $M$ and assume that their characteristic numbers $\lambda, \mu$ satisfy $\lambda<1,\; \lambda^2\mu>1$. According to \cite{IMSh}, the monodromy maps inside and outside of the polycyle are $C^1$-conjugated to the standard maps $x\mapsto C_1x^\lambda$ and $x\mapsto C^{-1}_2 x^{\lambda^2\mu}$. The points of intersection of the winding separatices with the semi-transversals will be denoted by $B_1$ and $B_2$.\footnote{Clearly, the choice of $B_1, B_2$ is not unique. The choice of $C_1, C_2$, however, is more subtle than expected and is, in a sense, `synchronized' with the lemmas below; see \cite{IMSh} for details.}

Since ``tears of the heart'' have three separatrix connections, this polycycle does not appear in generic two-parameter families, and it is natural to study and classify its 3-parameter unfoldings. However, such an unfolding contains special one-parameter subfamily that captures all known numeric invariants.
 
\begin{defin}
The \emph{standard one-parameter family} of ``tears of the heart'' is defined as a subfamily in which the separatrix connections corresponding to the `heart' remain intact, while only the `tear' connection (i.e., the separatrix loop of~$L$) is broken.
\end{defin}

Although the original results of \cite{IKS, GK, IMSh} describe invariants of classification with respect to {\it moderate} topological equivalence, we shall adopt the {\it weak} equivalence throughout this paper. Indeed, an inspection of the original proofs reveals that the arguments remain valid for the case of weak equivalence, since the core problem reduces to the existence of a suitable homeomorphism between the parameter bases\footnote{It was already noted in \cite{IKS}.}.

\begin{defin}
Two families of vector fields $\{v_\alpha \mid \alpha \in B \}$, $\{\tilde{v}_{\tilde{\alpha}} \mid \tilde{\alpha} \in \tilde{B}\}$ on $M$ are called \emph{weakly topologically equivalent} if there exists a map
\begin{equation}\label{eq:weak_eqiv}
H : B \times M \to \tilde{B} \times M, \quad H(\alpha, x) = \bigl(h(\alpha), H_\alpha(x)\bigr)
\end{equation}
such that $h : B \to \tilde{B}$ is a homeomorphism, and for each $\alpha \in B$ the map 
$H_\alpha : M \to M$ is a homeomorphism that links the phase portraits of $v_\alpha$ and $\tilde{v}_{h(\alpha)}$. 
\end{defin}

It is clear that the homeomorphism $h$ that establishes the equivalence of 3-parameter bases has to take the standard one-parameter subfamily of the first unfolding into the standard subfamily of the second.

It was recently shown in \cite{IMSh} that for \emph{topologically} generic families, there are at least four invariants of topological classification. As noted above, all these invariants arise from standard one-parameter subfamilies: it suffices to consider the standard one-parameter subfamilies to see that there is no equivalence when the invariants do not coincide.

The following result is in drastic contrast with the previous facts.

\begin{theorem} [on Diophantine families] \label{thm:main}
There exists a set $\mathcal{A}$ of full Lebesgue measure in the space of coefficients $(\lambda ,\mu ,B_{i},C_{i}) \simeq\mathbb R^6_+$, such that, for any two standard one-parameter families $\{v_\varepsilon\}$ and $\{\tilde {v}_{\tilde\varepsilon}\}$ with coefficients in~$\mathcal{A}$, weak equivalence is determined by the coincidence of two numeric invariants. More precisely, assuming the original vector fields that correspond to the unperturbed ``tears of the heart'' polycycle are orbitally topologically equivalent $(v_{0}\approx \tilde {v}_{0}$) and the coefficient tuples for both unfoldings are from~$\mathcal{A}$, the standard one-parameter subfamilies are weakly equivalent if and only if
$$
A=\tilde A
$$
and
$$
\tau=\tilde \tau \mod(1,A),
$$
where $A=-\ln\lambda/\ln(\lambda^2\mu)$, $\tau=s/\ln(\lambda^2\mu)$, $s= \ln \left(\frac {\ln C_1}{1 - \lambda} - \ln B_1\right)-\ln \left(\frac {\ln C_2}{1 - (\lambda^2\mu)^{-1}} - \ln B_2\right)$, and the tilde indicates the parameters associated with the second family.
\end{theorem}

Henceforth, we assume throughout that $A$ is irrational. For this case, the ‘only if’ part was proven in \cite{GK,IKS}. Note that this assumption does not change the measure of $\mathcal{A}$, since the set corresponding to rational $A$ has measure zero. 
The `if' part is the main result of this paper.

Taken together, Theorem~\ref{thm:main} and the results of~\cite{IMSh} reveal a drastic difference between the behavior of topologically and metrically generic families with ``tears of the heart''.

\begin{remark} We should stress that Theorem~\ref{thm:main} holds exclusively under the condition of precisely one inner and one outer separatrix winding. The presence of additional separatrices introduces new invariants (see \cite{GK}). We conjecture, however, that in this case within our framework of Diophantene families (see below) no new numeric invariants exist beyond those presented in \cite{GK, IKS}, although some combinatorial invariants may arise.
\end{remark}

\begin{remark} Note that while the statement of the theorem applies only to standard one-parameter subfamilies, the aforementioned `discrepancy' extends to the entire three-parameter unfoldings.
\end{remark}

\begin{remark} The four aforementioned invariants for Liouvillian (more precisely, so-called shift-exp-Liouvillian) families are $\lambda$, $\mu$, $\tau\mod{1,A}$, and the logarithm of the relative scale coefficient (introduced in~\cite{IMSh}) modulo~$\ln{\lambda}$.
\end{remark}

\section{LMF graphs}

The LMF graphs (Leontovich, Mayer, Fedorov graphs) are a convenient tool to prove that two vector fields on~$\mathbb S^2$ are orbitally topologically equivalent. In describing the LMF graphs in this section, we follow the exposition from~\cite{GI}.

Informally, an LMF graph is a way to represent the topological information visible in a phase portrait.
We present the definition for the case when the vector field has only hyperbolic singularities and omit some details needed in the general case (see~\cite{GI}). We also assume that our vector fields have finitely many cycles.
To define the LMF graph, we need to surround each source and sink of the vector field $v$ by a small transversal loop; we also choose transversal loops near monodromic sides of polycycles and on each side of limit cycles. The specific choice of these loop does not affect the LMF graph up to isotopy. If such a loop intersects with the separatrices of saddles, we mark the points of intersection and refer to them as \emph{truncation vertices}, because the separatrices will be truncated at these to avoid having edges of infinite length (e.g., when a separatrix winds onto a cycle). 

Now, the LMF graph of the vector field $v$ on $\mathbb S^2$ with hyperbolic singularities is a graph
$LMF(v)$, embedded in $\mathbb S^2$, that consists of the following elements.
\smallskip

{\bf Vertices}:
(1) all singular points of $v$;
(2) all truncation vertices: the intersection points of the separatrices of~$v$
with the transversal loops chosen above;
(3) a point at each cycle;
(4) a point at each empty transversal loop, i.e., at a transversal loop that does not intersect the separatrices of~$v$.

{\bf Edges:}
(1) unstable (stable) separatrices of saddles, provided their $\omega$-
(resp., $\alpha$-) limit sets consist of a single singular point that is not a focus;
(2) truncated unstable (stable) separatrices of saddles — i.e., arcs of separatrices, each connecting a saddle and a truncation vertex, — if these separatrices do not satisfy the condition of case 1);
(3) limit cycles (such an edge starts and ends at vertices of type~3);
(4) pieces of transversal loops between adjucent truncation vertices, or the whole empty transversal loops (these are incident with vertices of type~4).

{\bf Orientation.}
The orientation of edges of types 1, 2, 3 is induced by time parametrization. The orientation of edges of type 4 is counterclockwise
with respect to the disk that contains the $\alpha$- or $\omega$-limit set corresponding to the transversal loop.

{\bf Labeling.}
Each vertex of the LMF graph is labeled by the description of its type, namely the
labels say Singular Point (SP), Truncation Vertex (TV), Vertex on a
Limit Cycle (VLC), Vertex on an Empty Transversal Loop (VETL).
Similarly, the labels on the edges say Stable Separatrix (SS), Unstable
Separatrix (US), Separatrix Connection (SC), Stable Truncated Separatrix (STS), Unstable Truncated Separatrix (UTS), Limit Cycle (LC),
Outgoing Transversal Loop (OTL), Ingoing Transversal Loop (ITL). We say that a transversal loop
is ingoing if this is a loop around its $\omega$-limit set; otherwise we say that
the transversal loop is outgoing.

\bigskip

Fedorov's theorem below (based on the previous result of Andronov, Leontovich, Gordon, Mayer \cite{ALGM}) provides an LMF graph based criterion for the conjugacy of vector fields on the sphere. 

\begin{theorem} [R. Fedorov, \cite{F}] If two LMF graphs 
$LMF(v),LMF(w)$ of two vector fields $v, w$ are isotopic on the sphere (i.e., there
exists an orientation-preserving homeomorphism of the sphere which
maps one to another, preserves orientation on edges and matches labels on edges and vertices), then $v$ and $w$ are orbitally topologically
equivalent.
\end{theorem}
Effectively, this theorem bridges the gap between geometric intuition and formal proof: it ensures that fields which appear conjugate `from the picture' are indeed topologically conjugate.

\section{The set of Diophantine families}

Let $A$ and $s$ be as in Theorem~\ref{thm:main}, and let $\gamma=\ln(\lambda^2\mu)$.

\begin{defin} A family is called \emph {Diophantine} if 
the inclusion
$$
A-m/n\in \left[ \frac{s - 1/(m^2+n^2)}{\gamma n}, \frac{s+ 1/(m^2+n^2)}{\gamma n}\right]
$$
holds only for a finite number of integer pairs $(m,n)$. 
\end{defin}

The full measure set $\mathcal{A}$ in Theorem~\ref{thm:main} is the set of coefficients of Diophantine families with additional assumption that $A = \frac{-\ln \lambda}{\ln \lambda^2\mu}$ is irrational.
Let us verify that Diophantine families correspond to a set of full Lebesgue measure in the coefficient space; the same will trivially follow for the set $\mathcal{A}$.

\begin{proposition} For almost all parameters $(\lambda ,\mu ,B_{i},C_{i})$ with respect to the Lebesgue measure, the corresponding families are Diophantine.
\end{proposition}

\proof For fixed $\lambda^2\mu, B_i$, $C_2$ and $\ln C_1/(1-\lambda)$, the values of $\gamma$ and $s$ are likewise fixed, while $A$ can be varied by altering $\lambda$ and $\mu$ (while $\lambda^2\mu$ does not change). For arbitrary $T>0$, let us show that, for any fixed $\gamma$ and $s$, the Diophantine property holds for almost every $A\in[-T,T]$. Indeed, put 
$$J=1+(|s|+1)/|\gamma|.$$
\smallskip
\noindent {\bf Case 1.} If $|m| > (|A|+J)|n|$, it follows that $|A-m/n|>J$. Since  
$$
\left[ \frac{s - 1/(m^2+n^2)}{\gamma n}, \frac{s+ 1/(m^2+n^2)}{\gamma n}\right]\subset[-J,J],
$$
no solutions exist in this case.
\smallskip

\noindent {\bf Case 2.} 
Another option is $|m|\leq(|A|+J)|n|\leq (T+J)|n|$.

For any pair $(m,n)$, the inclusion holds for~$A$ that form a segment of length $\frac{2}{\gamma n(m^2+n^2)}$. 
Therefore, for fixed $n$ and all possible $m$ the total measure of the union $J_n$ of such segments can be estimated as

$$
\sum_{|m|\leq (T+J)|n|} \frac{2}{\gamma n(m^2+n^2)}\leq\frac{2}{\gamma |n|} \sum_{|m|\leq (T+J)|n|} \frac{1}{n^2}\leq\frac{2}{\gamma n}\cdot \frac{2(T+J)}{n}.
$$

Thus, the total measure of the union of $J_n$ over $|n|>N$ is at most
$$
\frac{4(T+J)}{\gamma}\sum_{|n|>N}\frac{1}{n^2}=o(1), \quad n \to +\infty.
$$

If the Diophantine property does not hold for $A$, then $A$ belongs to $J_n$ for infinitely many~$n$, meaning it lies in the tail sets $\cup_{|n|>N}J_n$ for any~$N$. Since the measure of these tail sets tends to zero as $N\rightarrow \infty$, it follows that the set of forbidden values of $A$ is of measure zero.
\smallskip

Consequently, for fixed $\gamma$ and $s$, the Diophantine property holds for almost every $A$ with respect to the Lebesgue measure. When we fix the values of $\lambda^2\mu, B_1, B_2, C_2$, and $C_1/(1-\lambda)$, we obtain a curve in the space of coefficients. When this curve is parametrized by $A$, almost every value of the parameter corresponds to a Diophantine family. Since the mapping $\lambda \mapsto \ln \lambda /\ln (\lambda ^{2}\mu )$ is a diffeomorphism (recall that $\lambda^2\mu>1$ is fixed), the same is true when we reparametrize the curve by $\lambda$. Fubini's theorem applied to the foliation of the coefficients space into these curves completes the proof: for almost every tuple of coefficients, the family is Diophantine.
\qed

\section{Proof of the Main Theorem}

\subsection{Sparkling saddle connections and useful lemmas}

Consider a standard one-parameter family with an inner saddle $I$ inside the loop. Then for an infinite sequence of values $\varepsilon = \varepsilon_n \to 0$ sparkling saddle connections between $L$ and $I$ occur, i.e., a vector field corresponding to $\varepsilon_n$ exhibits a separatrix connection that ``turns around the vanished separatrix loop $n$ times''. The same holds for $L$ and the outer saddle $E$. We assume, without loss of generality, that these bifurcations occur for $\varepsilon>0$, by applying the change $\varepsilon \mapsto -\varepsilon$ if necessary.

In the double-logarithmic chart $\ln(-\ln\varepsilon)$, the bifurcation parameters form two perturbed arithmetic progressions $(e_n)_{n\in\mathbb N}$ and $(i_m)_{m\in\mathbb N}$; $e_n$ correspond to the $EL$ connections. It is worth noting that a pair of perturbed arithmetic progressions cannot always be mapped into another such pair by a homeomorphism of the real line (see \cite{IKS, IMSh}). This motivates the following definition.

\begin{defin} We say that two collections of marked sequences of points on $\mathbb R_+$  are \emph{order-equivalent} at 0, if there exists a homeomorphism of a neighborhood of $0$ that maps one collection onto the other, preserving the marks, i.e., each sequence is taken into a corresponding sequence of another collection, with a possible exception at finitely many points.
\end{defin}

The following lemma is based on the results of~\cite{IMSh}.

\begin{lemma} \label{proper} In the case where $A=\tilde A$ and $\tau=\tilde\tau \mod(1,A)$, the sequences corresponding to $LI$ connections and  $LE$ connections are order-equivalent for Diophantine families.
\end{lemma}

\proof 
First, assume that  $LI$ and $LE$ connections do not occur simultaneously. As was shown in the proof of Lemma~4.1 of~\cite{IMSh}, for two standard subfamilies, the order equivalence of the pairs of sequences corresponding to $LI$ and  $LE$ connections is equivalent to a condition that $A$ does not belong to a certain family of intervals. More precisely, there exist functions $Q_i(m, n), \; i = 1, 2$, defined in terms of asimptotic expansions for the saddle connection parameter values $i_n, e_m, \tilde{i}_n, \tilde{e_m}$, such that the pairs of sequences of $LI$ and $LE$ connections are order-equivalent if and only if the condition
$$
A-m/n\notin \frac{s+(Q_1(m,n),Q_2(m,n))}{\gamma n}
$$
holds for all sufficiently large $m$ and~$n$ (see formula~(4.2) of~\cite{IMSh}; the explicit form of $Q_i(m,n)$ is given right after this formula). 
The important part is that $Q_i(m,n)$ decay exponentially in $(m,n)$; their leading terms were explicitly calculated in~\cite[Prop.~4.1]{IMSh}. 

For Diophantine families, the inclusion
$$
A-m/n\in \left[ \frac{s - 1/(m^2+n^2)}{\gamma n}, \frac{s+ 1/(m^2+n^2)}{\gamma n}\right]
$$
holds only for a finite number of integer pairs $(m,n)$. Since $|Q_i(m,n)|$ decays exponentially, it follows that $|Q_1(m,n)|+|Q_2(m,n)|<1/(m^2+n^2)$ holds asymptotically, which completes the proof of the lemma for this case.

According to the calculations in \cite{IMSh} (see Lemma~3.1), $LI$ and $LE$ connections occur simultaneously only if
$$
n\ln\lambda-m\ln(\lambda^2\mu)=s - r_1(n) + r_2(m),
$$
where $r_1(n) = O(\lambda^{n})$ and $r_2(m) = O((\lambda^2\mu)^{-m})$. This can be written as
$$
A-m/n=s/(\gamma n)+O(\lambda^{n})+O((\lambda^2\mu)^{-m}),
$$
 which is not possible for Diophantine families provided that $(m,n)$ is sufficiently large. The proof is complete.
\qed

In principle, new numeric invariants could arise from other saddle connections in a standard one-parameter family. The following two lemmas, however, rule out this possibility.

\begin{lemma}\label{lem:2}
    There are no sparkling saddle connections in a standard one-parameter family other than $LI$ connections, $LE$ connections and $EI$ connections.
\end{lemma}

\proof In a standard one-parameter family, only four separatrices can form a connection: two separatrices of~$L$, one of~$I$, and one of~$E$. A connection is always formed by an incoming and an outgoing separatrix; hence, only four options are available. However, the two separatrices of $L$ cannot form a connection because for $\varepsilon > 0$ they form a Bendixson sack, i.e., a trapping region. This leaves the three options listed in the lemma.
\qed

\begin{lemma}\label{lem:3} If the $LE$ and $LI$ connections of two families are order-equivalent, then the same holds for the collections of all $LE$, $LI$, and $EI$ connections of the two families.
\end{lemma}

This lemma formally states that $EI$ connections yield no additional invariants in the case of a standard one-parameter family. The proof is postponed until Section~\ref{sec:L3_proof}.

\subsection{No Additional Combinatorics Lemma}

Before the proof of Theorem~\ref{thm:main}, we need a final lemma that rules out the existence of non-trivial {\it combinatorial} invariants:

\begin{lemma} Suppose that the stable separatrix of the saddle $I$ is the only separatrix that winds onto the loop of the `tears of the heart' in the reversed time and the unstable separatrix of $E$ is the onply one that winds onto the whole polycycle. Then both unstable separatrices of $I$ share the same attractor and both stable separatrices of~$E$ share the same repellor.
\end{lemma}

\proof Suppose, to the contrary, that the two unstable separatrices of the saddle $I$ tend to distinct attractors. It follows that the winding stable separatrix of~$I$ is contained in the common boundary of their basins of attraction, which are open. Consider a strip between the $n$-th and $(n+1)$-st turns of the separatrix near the loop. Within this strip, both basins are present and, therefore, must be separated by another separatrix. Since $n$ was arbitrary, this `new' separatrix is winding towards the loop. This contradicts the assumption that there is a unique winding separatrix inside the loop. Hence, both separatrices must tend to the same attractor.
The argument for the saddle $E$ is analogous.
\qed

\subsection{The Proof of Theorem~\ref{thm:main}}    

In order to prove Theorem~\ref{thm:main}, we need to establish that the LMF graphs of the two families are isotopic for all values of~$\varepsilon$. First, we treat the case $\varepsilon< 0$. Here, an attracting hyperbolic limit cycle appears outside the polycycle, whereas a repelling one emerges inside the loop (see, e.g., \cite[p.~254, Fig.~143]{ALGM}). The two unstable separatrices (namely, those of $L$ and $E$) tend to the attracting cycle. Similarly, the stable separatrices of $L$ and $I$ spiral away from the repelling one.

Combined with the structural stability of the `outer' part of the vector fields in the family, this implies that the LMF graphs corresponding to $\varepsilon<0$ are isotopic. Therefore, Fedorov’s theorem establishes the weak equivalence of the families restricted to $\{\varepsilon \leqslant 0\}$, where the identity (or, in fact, any) homeomorphism of the base $\{\varepsilon \leqslant 0\}$ can be taken as the parameter mapping.

It remains to treat the case $\varepsilon>0$. According to Lemma~\ref{proper}, $LE$ and $LI$ connections are order-equivalent for Diophantine families with $A=\tilde A$ and $\tau=\tilde \tau\pmod {1,A}$. Hence, by Lemma~\ref{lem:3}, the collections of $LE$, $LI$, and $EI$ connections are also order-equivalent for the two families. Let $h$ denote the homeomorphism that realizes this equivalence of the bases $\{\varepsilon \geqslant 0\}$. 

First, assume that no saddle connection occurs for $v_{\varepsilon}$. In this subcase, the unstable separatrices of $L$ and $E$ tend to the attractor of the unstable separatrices of $I$ by continuity (note that this attractor is unique according to Lemma 4, and the basin of attraction is open due to hyperbolicity). Similarly, the stable separatrices of $L$ and $I$ tend to the repellor of the stable separatrices of $E$ under time reversal. Note that all other vertices and edges of the LMF graph arise from structurally stable components of the vector field $v_0$. Consequently, all LMF graphs in this subcase are isotopic, and the vector fields $v_{\varepsilon}$ and $\tilde {v}_{h(\varepsilon)}$ are topologically equivalent.

Conversely, suppose that there exists a connection for $v_\varepsilon$. By Lemma~\ref{lem:2}, it must be an $EI$, $LE$, or $LI$ connection. Then (by the construction of $h$) a connection of the same type also occurs for $\tilde v_{h(\varepsilon)}$. Consequently, the LMF graph for this subcase is obtained from the LMF graph of the previous subcase via a graph surgery that is identical for both fields $v_\varepsilon$ and $\tilde v_{h(\varepsilon)}$. Therefore, the LMF graphs remain isotopic for $v_\varepsilon$ and $\tilde v_{h(\varepsilon)}$ as well.

Since all cases have been exhausted, the proof is complete.

\section{Proof of Lemma~\ref{lem:3}}\label{sec:L3_proof}

\subsection{An explicit description of the ordering}

We now provide the proof of Lemma~\ref{lem:3}, which was deferred from the previous section. Recall that Lemma~\ref{lem:3} states that if the $LE$ and $LI$ connections of two families are order-equivalent, then so is the collection of all $LE$, $LI$, and $EI$ connections.

The following proposition provides an explicit description of the arrangement of $EI$ connections.

\begin{proposition} \label{prop:2} The bifurcation parameters corresponding to LE and LI connections partition the parameter line into a sequence of intervals. For sufficiently small $\varepsilon>0$, any such open interval contains exactly one EI connection, regardless of the types of connections at its endpoints.
\end{proposition}

With this proposition in hand, we can prove Lemma~\ref{lem:3}.

\proof[of Lemma 3]
By the hypothesis of the lemma, we are given a homeomorphism of the parameter line that maps $LE$ connections to $LE$ connections and $LI$ connections to $LI$ connections. Consequently, it maps the open intervals between them to the corresponding intervals in the image. By Proposition 2, each such interval — both in the domain and the image — contains exactly one $EI$ connection. It is then a simple matter to modify the homeomorphism on each interval, keeping the endpoints fixed, such that the $EI$ connection of the domain is mapped to the $EI$ connection of the image. This adjusted homeomorphism yields the desired order equivalence, concluding the proof.
\qed

\subsection{Proof of Proposition~\ref{prop:2}}
 Let $\Gamma$ be a transversal section passing through a point of the saddle loop of the `tears of the heart' polycycle. The transversal $\Gamma$ does not depend on the parameter~$\varepsilon$. When the loop is unfolded for $\varepsilon > 0$, a segment between the first points  of intersection of $\Gamma$ with the stable and the unstable separatrices of $L$ appears on $\Gamma$. We refer to this segment as \emph{`the gap'}. Note that for $\varepsilon > 0$ the unstable separatrix of the saddle $E$ intersects the gap at exactly one point, and the same is true for the stable separatrix of the saddle~$I$; moreover, an $EI$ connection occurs if and only if these two points coincide.

We will consider the point where a separatrix of $E$ enters the gap in a parameter-dependent natural chart on $\Gamma$, that is, a chart where $x=0$ corresponds to the intersection with the stable separatrix of~$L$. This point is obtained via iterating the monodromy map $\mathcal{F}_\varepsilon$ of the perturbed `tears of the heart' polycycle. An analogous procedure for the entry point of the interior saddle~$I$ will require another natural chart, the one where $y = 0$ corresponds to the point of the unstable separatrix of~$L$, and the monodromy along the perturbed saddle loop will be considered in the reversed time and denoted~$\mathcal{G}_\varepsilon$. Since we are free to choose the natural coordinates that will be used in the argument, we will assume that they are related by the following coordinate change formula:
$$
y = -\rho(\varepsilon)-x.
$$
Here $\rho(\varepsilon) \ge 0$ is the size of the gap when measured in any of our two natural coordinates. Note that for both coordinates $x$ and $y$, the gap is at the negative semi-axis, see Fig.~\ref{fig:gap}.

\begin{figure}[!ht]
\begin{center}
\includegraphics{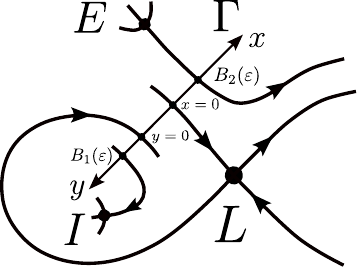}
\caption{The gap and the coordinates $x$ and~$y$ at the transversal $\Gamma$.} \label{fig:gap}
\end{center}
 \end{figure}

Assume that for a particular value of $\varepsilon$ a segment of the unstable separatrix of $E$ that starts at $B_{2}(\varepsilon)$ makes $n$ winds around the polycycle before it reaches the gap. Similarly, assume that the stable separatrix of $I$ makes $k$ turns around the perturbed loop when it goes from~$B_{1}(\varepsilon)$ to the gap in the reversed time. Recall that $\mathcal{F}_\varepsilon$ is the monodromy map of the perturbed `tears of the heart' in the positive time and $\mathcal{G}_\varepsilon$ is the monodromy map of the perturbed loop in the reversed time. Each map is defined on a parameter dependent semitransversal contained in $\Gamma$. The points where the separatrices of $E$ and $I$ enter the gap are $\mathcal{F}^n_\varepsilon(B_2(\varepsilon))$ and $\mathcal{G}^k_\varepsilon(B_1(\varepsilon))$, respectively. We consider the difference between the positions of these two points expressed in the $y$-coordinate:
\begin{equation}\label{eq:d}
d_{n,k}(\varepsilon) = y(\mathcal{G}^k_{\varepsilon} (B_1(\varepsilon))) - y(\mathcal{F}^n_{\varepsilon}(B_2(\varepsilon))) = y(\mathcal{G}^k_{\varepsilon} (B_1(\varepsilon))) + x(\mathcal{F}^n_{\varepsilon}(B_2(\varepsilon))) + \rho(\varepsilon).
\end{equation}

\begin{proposition} The function $d_{n,k}(\varepsilon)$ is monotonically decreasing provided $n$ and $k$ are sufficiently large.
\end{proposition}
\proof
We denote by $F_\varepsilon$ the map $\mathcal{F}_\varepsilon$ written in the $x$-chart, and by $G_\varepsilon$ the map $\mathcal{G}_\varepsilon$ written in the $y$-chart. With a slight abuse of notation, below we write $F_\varepsilon(B_2(\varepsilon))$ instead of $F_\varepsilon(x(B_2(\varepsilon)))$ and $G_\varepsilon(B_1(\varepsilon))$ instead of $G_\varepsilon(y(B_1(\varepsilon)))$, making the transition to the corresponding chart implicit.

We will show that the derivative of $F_{\varepsilon}^{n}(B_{2}(\varepsilon))$ in the natural chart satisfies an estimate
\begin{equation}
\partial_{\varepsilon}F_{\varepsilon}^{n}(B_{2}(\varepsilon)) < -\frac{3}{4}\rho^{\prime}(\varepsilon) < 0.
\label{est}
\end{equation}
The same estimate for $G_{\varepsilon}^{k}(B_{1}(\varepsilon))$ is obtained analogously. This immediately implies that
$$
d_{n,k}'(\varepsilon) = \partial_{\varepsilon}G_{\varepsilon}^{k}(B_{1}(\varepsilon)) +
\partial_{\varepsilon}F_{\varepsilon}^{n}(B_{2}(\varepsilon))+\rho^{\prime}(\varepsilon)<
-\rho^{\prime}(\varepsilon)/2 < 0,
$$
which establishes the monotonicity.

We now proceed to the proof of estimate~\eqref{est}. First, note that $F_{\varepsilon}(0)= -\rho(\varepsilon)$: the monodromy takes one endpoint of the gap into the other, and the gap appears on the negative semi-axis w.r.t the chosen natural coordinates. Second, we will need two properties of the map $F$ that follow from Lemma 5 and Corollaries 2 and 3 of \cite{IKS} (see also\footnote{Note that, in contrast to \cite{IKS}, where the authors take $\rho (\varepsilon)$ as a new parameter, we perform our analysis in terms of the original~$\varepsilon$. The reparameterization in \cite{IKS} is valid precisely due to the non-zero speed condition mentioned in the first footnote, which ensures that $\rho^{\prime }(\varepsilon)>0$.} the proof of Lemma~6 in~\cite{IKS}):
\begin{itemize}
    \item For any $\delta \in (0,1)$, for sufficiently small $|\varepsilon|$ and $|x|$, the map $F_\varepsilon(\cdot)$ is $\delta$-contractive;
  \item The derivative $\partial_{\varepsilon}F_\varepsilon(x)$ is  $\partial_{\varepsilon}F_{\varepsilon}(0)|_{\varepsilon = 0}+o(1)$  as  $(\varepsilon,x) \to (0,0)$, where $\partial_{\varepsilon}F_{\varepsilon}(0)=-\rho^{\prime}(\varepsilon)<0$.
\end{itemize}

Let  $R$ be a small rectangular neighborhood of the origin in the $(\varepsilon, x)$-plane where we have that $\partial_{\varepsilon}F_{\varepsilon}(x) < -\frac{9}{10}\rho'(0)$ and each $F_\varepsilon$ is $\delta$-contractive with $\delta ={1}{/10}$. Denote $K = \frac{9}{10}\rho'(0)$.

For sufficiently small $\varepsilon$ the total number of winds $n$ is large, whereas the number of winds it takes to get from $B_2(\varepsilon)$ to $R$ is bounded. Hence, we can fix $r \in \mathbb{N}$ such that $F^r_\varepsilon(B_2(\varepsilon))$ and all subsequent points of the orbit up to $F_{\varepsilon}^{n-1}(B_2(\varepsilon))$ are in $R$ and we can decompose $n = t + r$ and $F^{n}_\varepsilon = F^t_\varepsilon \circ F^r_\varepsilon$. Since $F_{\varepsilon}^{r}(x)$ is smooth with respect to $(x,\varepsilon)$ in a neighborhood of $(B_2(\varepsilon), 0)$, the expression $|\partial_{\varepsilon}F_{\varepsilon}^{r}(B_2(\varepsilon))|$ is uniformly bounded, say, by a constant $M>0$.

By the chain rule,
$$
\partial_\varepsilon F_{\varepsilon}^{r+1}(B_2(\varepsilon))=\partial_\varepsilon F_\varepsilon+\partial_x F_\varepsilon\cdot\partial_\varepsilon F_{\varepsilon}^{r}(B_2(\varepsilon)) < -K + \delta M.
$$
We omit the argument $F_{\varepsilon}^{r}(B_2(\varepsilon))$ of the derivatives of $F_\varepsilon$ for brevity.
For the next iterate, we analogously have
\[
\partial_\varepsilon F_{\varepsilon}^{r+2}(B_2(\varepsilon))=\partial_\varepsilon F_\varepsilon+\partial_x F_\varepsilon\cdot\partial_\varepsilon F_{\varepsilon}^{r+1}(B_2(\varepsilon)).
\]
If the estimate $a_1 = -K + \delta M$ on $\partial_\varepsilon F_{\varepsilon}^{r+1}(B_2(\varepsilon))$ is negative, the second term is negative and we can write
\[
\partial_\varepsilon F_{\varepsilon}^{r+2}(B_2(\varepsilon)) < -K.
\]
If $a_1$ is positive, we have an estimate
\[
\partial_\varepsilon F_{\varepsilon}^{r+2}(B_2(\varepsilon)) < -K + \delta a_1 = a_2.
\]
Let $(a_j)_{j\ge 1}$ be the positive semi-orbit of $M$ for the affine map $\xi\colon u \mapsto -K + \delta\cdot u$. Arguing inductively, we obtain that while $a_{j-1} > 0$, the point $a_j$ serves as an estimate from above on the corresponding derivative:
\[
\partial_\varepsilon F_{\varepsilon}^{r+j}(B_2(\varepsilon)) < a_j.
\]
The affine map $\xi$ has an attracting fixed point at $\frac{-K}{1-\delta}$. Since $K = \frac{9}{10}\rho'(0)$ and $\delta = \frac{1}{10}$, we have $\frac{-K}{1-\delta} = -\rho'(0) < 0$. Hence, we eventually obtain the first $j$ such that $a_{j-1} < 0$. Then, as above, for all $i \ge j$ we get an estimate
\[
\partial_\varepsilon F_{\varepsilon}^{r+i}(B_2(\varepsilon)) < -K.
\]
This holds, in particular, for $i = t$, which implies the desired estimate 
$$
\partial_\varepsilon F_\varepsilon^n(B_2(\varepsilon)) < -\frac{3}{4}\rho'(\varepsilon).
$$
\qed

 Let us return to the proof of Proposition~\ref{prop:2}. As mentioned above, in the $\ln (-\ln(\varepsilon))$ chart, the parameter values of the $LE$ and $LI$ connections form two perturbed arithmetic progressions (where the perturbation decays exponentially as the parameter $\ln(-\ln (\varepsilon))$ increases; see \cite{IMSh}). Let $a_{n}$ denote the progression with the smaller common difference, and $b_{n}$ the one with the larger common difference.\footnote{Note that the common differences of the progressions cannot coincide since $A\notin \mathbb Q$.} One of the functions $G_{\varepsilon}^{n}(B_{1}(\varepsilon))$ and $F_{\varepsilon}^{n}(B_{2}(\varepsilon))$ is continuous on $[a_{n},a_{n+1})$, while the other is continuous on $[b_{n},b_{n+1})$. Let us consider the case where the $EL$ connection values $(e_n)$ play the role of $(a_n)$; the opposite case is analogous.

\begin{figure}[!ht]
\begin{center}
\includegraphics{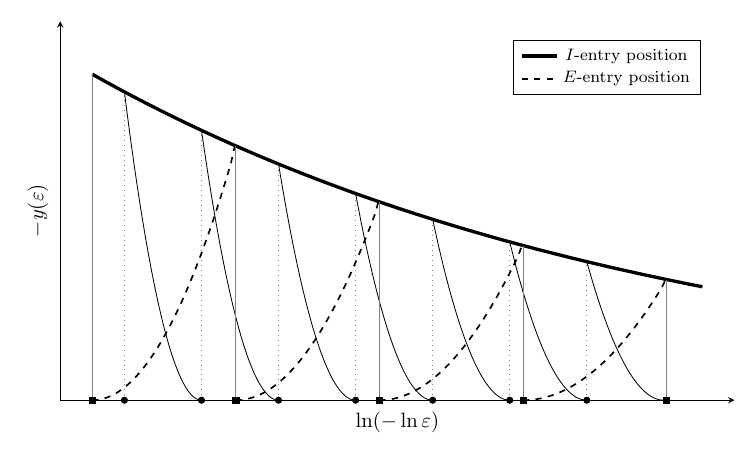}
\caption{Qualitative plots of the points where the $I$- and $E$-separatrices reach the gap.  Due to the monotonicity of the difference $d_{n,k}$, the intersection points (which correspond to $EI$ connections) are positioned as described by Proposition~\ref{prop:2}.} \label{fig:trap}
\end{center}
 \end{figure} 
 
Consider a curvilinear trapezoid with a horizontal base $[e_j, e_{j+1}]$ at the $\ln(-\ln\varepsilon)$-axis, two vertical sides, and with upper boundary being an arc of the graph of the gap length as a function of $\ln(-\ln\varepsilon)$. If $[e_j, e_{j+1})$ contains no points of the sequence $(i_{n})$, it is contained in some $[i_k, i_{k+1})$, and the graph of $F^j_\varepsilon(B_2(\varepsilon))$ connects the opposite corners of the trapezoid while the graph of $G^k_\varepsilon(B_1(\varepsilon))$ connects some points on the two vertical sides. Then the difference $d_{j,k}(\varepsilon)$ from~\eqref{eq:d} admits at least one zero by the intermediate value theorem and has at most one zero by the monotonicity of $d_{j,k}$. This zero corresponds to an $EI$ connection. If the interval $[e_j, e_{j+1})$ contains an element of $(i_n)$, such an element $i_k$ is unique, provided that $\varepsilon$ is sufficiently small. The line $\{\ln(-\ln \varepsilon)=i_k\}$ partitions the trapezoid into two sub-regions, where the same reasoning as above can be applied, see Fig.~\ref{fig:trap}. This completes the proof of Proposition~\ref{prop:2}.

\section{Conclusion}

The Main Theorem and the results of \cite{IMSh} establish a gap between the topological and metric approaches to the classification of generic families of vector fields. This discrepancy is shown to be rooted in the surprising arithmetical nature of the problem, reflecting the fundamental contrast between the small-denominator-like properties inherent to Liouville-type parameters and the regularity characteristic of metrically typical ones.

The foregoing suggests that topologically generic vector fields with ``tears of the heart'' may give rise to new series of numerical invariants, whereas the metrically generic  Diophantine case holds the promise of a complete classification. One should keep in mind, however, that even in the latter setting, combinatorial invariants may emerge. For instance, any answer to the question of whether inequivalent links or knots arise within the bifurcation diagrams of the full three-parameter families with `tears of the heart' would be of considerable interest.


\medskip
\noindent
{\large \bf Yu. S. Ilyashenko,\\}
HSE University, Moscow, Russia\\
E-mail: \texttt{\tt yulijs@gmail.com}

\medskip
\noindent
{\large \bf Stanislav Minkov,\\}
Brook Institute of Electronic Control Machines, Moscow, Russia\\
E-mail: \texttt{stanislav.minkov@yandex.ru}

\medskip
\noindent
{\large \bf Ivan Shilin,\\}
HSE University, Moscow, Russia\\
E-mail: \texttt{i.s.shilin@yandex.ru}

\end{document}